\documentclass[reqno]{amsart}
\usepackage{amssymb,amsmath,amsthm,amscd,latexsym,amsfonts}
\usepackage{mathtools}
\usepackage[english]{babel}
\usepackage[T1]{fontenc}
\usepackage{inputenc}
\usepackage[arrow,matrix]{xy}
\usepackage{amsaddr}
\usepackage{graphicx}
\usepackage{cite}
\newtheorem{thm}{Theorem}[section]
\newtheorem{cor}[thm]{Corollary}
\newtheorem{lem}[thm]{Lemma}
\newtheorem{prop}[thm]{Proposition}

\theoremstyle{definition}
\newtheorem{defn}[thm]{Definition}

\theoremstyle{remark}
\newtheorem{rem}[thm]{Remark}
\numberwithin{equation}{section}
\DeclareMathOperator{\Der}{Der}

\DeclareMathOperator{\Ann}{Ann}

\DeclareMathOperator{\gr}{gr}
\DeclareMathOperator{\Center}{Center}

\begin{document}

\title[Solvable Leibniz
algebras with naturally graded non-Lie $p$-filiform nilradicals]{Solvable Leibniz algebras with naturally graded non-Lie $p$-filiform nilradicals and maximal complemented space of its nilradical}
\author{J.~Q.~Adashev\textsuperscript{1}, L.M.~Camacho\textsuperscript{2}, B.~A.~Omirov\textsuperscript{1,3}}

\address{\textsuperscript{1}Institute of Mathematics, Uzbekistan  Academy of  Sciences, 100170, Tashkent,  Uzbekistan, adashevjq@mail.ru}
\address{\textsuperscript{2}Dpto. Matem\'{a}tica Aplicada I. Universidad de Sevilla. Avda. Reina Mercedes, s/n. 41012
Sevilla. (Spain) E-mail address: lcamacho@us.es}
\address{\textsuperscript{3} National University of Uzbekistan, 4, University str., 100174, Tashkent,  Uzbekistan, omirovb@mail.ru}

\begin{abstract} The present article is a part of the study of solvable Leibniz algebras with a given nilradical. In this paper solvable Leibniz algebras, whose nilradicals is naturally graded $p$-filiform non-Lie Leibniz algebra  $(n-p\geq4)$ and the complemented space to nilradical has maximal dimension, are described up to isomorphism. Moreover, among obtained algebras we indicate the rigid and complete algebras.
\end{abstract}



\maketitle

\section{Introduction}

During the last decades the theory of Leibniz algebras has been actively investigated and many results of the Lie Theory have been transferred to Leibniz algebras.

Levi's decomposition asserts that every finite-dimensional Lie algebra is a semidirect sum of a semisimple Lie subalgebra and solvable radical \cite{jacobson}, while semisimple Lie algebras over the field of complex numbers have been classified by E. Cartan \cite{cartan} and over the field of real numbers by F. Gantmacher \cite{gant}. Thus, the problem description of finite-dimensional Lie algebras is reduced to the study of solvable Lie algebras. Till present the classification of solvable Lie algebras is known for dimensions not greater than six \cite{graaf}, \cite{pat}. Also there are several works devoted to the classification of solvable Lie algebras in an arbitrary finite-dimensions \cite{AnCaGa1, AnCaGa1x, AnCaGa2}, \cite{Mal,SnWi, TrWi}. In fact, there are solvable Lie algebras constructed using the method explained in \cite{Mub}.

For finite-dimensional Leibniz algebras over a field of zero characteristic, there is an analogue of Levi's decomposition: any Leibniz algebra is decomposed into a semidirect sum of a semisimple Lie algebra and its solvable radical \cite{Bar}. Therefore, similar to  Lie case, the main problem of the study of Leibniz algebras reduced solvable ones.

In the paper \cite{Nulfilrad}, the method that describes solvable Lie algebras by means of its radical is adapted for Leibniz case.

Since the description of finite-dimensional solvable Leibniz algebras is a boundless problem (even for solvable Lie algebras), new approaches are developing. Relevant tools of geometric approaches are properties of Zariski topology and the natural action of linear reductive group on varieties of algebras in a such way that orbits under the action consists of isomorphic algebras. It is a well-known result of algebraic geometry that any algebraic variety (evidently, algebras defined via identities form an algebraic variety) is a union of a finite number of irreducible components. The most important algebras are those whose orbits under the action are open sets in sense of Zariski topology (such algebra are called {\it rigid} algebras). The algebras of a variety with open orbits are important since the closures of orbits of such algebras form irreducible components of the variety. At the same time there exists an irreducible component which is not the closure of orbit of any algebra. This fact does not detract the importance of algebras with open orbits.

This was a motivation for many works focused to discovering of algebras with open orbits and to description of sufficient properties of such algebras \cite{Burde,O'Halloran1,O'Halloran2}.

The aim of this article is to describe solvable Leibniz algebras with naturally graded non-Lie $p$-filiform nilradicals and with maximal dimension of complemented space of its nilradical. Namely, in arbitrary finite dimension, we got three types of such algebras  ($R(\mu_1,k), \ R(\mu_2,k)$ and $R(\mu_3,k+2)$) and we established that the algebra $R(\mu_3,k+2)$ is complete and cohomologically rigid.

Throughout the paper we shall consider finite-dimensional vector spaces and complex algebras. Moreover, in the multiplication table the omitted products are assumed to be zero and we shall consider non-nilpotent solvable algebras (unless stated otherwise).

\section{Preliminaries}

We recall the necessary background in order to make the comprehensive paper.

\begin{defn}\cite{Loday} A  \textit{Leibniz algebra}  $L$ is a vector space over $\mathbb{F}$ equipped with a bilinear map (multiplication) satisfying the Leibniz identity
\[ \big[x,[y,z]\big]=\big[[x,y],z\big] - \big[[x,z],y\big] \]
for all $x,y,z\in L.$
\end{defn}

We refer readers to works \cite{Loday} and \cite{Loday1} for examples of Leibniz algebras.

Further we will use the following notation
\[ {\mathcal L}(x, y, z)=[x,[y,z]] - [[x,y],z] + [[x,z],y].\]

It is obvious that the identity ${\mathcal L}(x, y, z)=0$ determines the Leibniz algebras.

For a given Leibniz algebra $L$ we can define the following two-sided ideals
$$\Ann_r(L) =\{x \in L \mid [y,x] = 0,\ \text{for \ all}\ y \in L \},$$
$$\Center(L) =\{x \in L \mid [x,y]=[y,x] = 0,\ \text{for \ all}\ y \in L \}$$
called the \emph{right annihilator} and the \emph{center} of $L$, respectively. From the Leibniz identity we conclude that the following elements $[x,x], [x,y]+[y,x]$ in $\Ann_r(L)$ for any $x, y \in L$.

\label{der}  A linear map $d: L\rightarrow L$ of a Leibniz algebra $L$ is said to be a {\it derivation} if
for all $x, y\in L,$ the following condition holds:
\begin{equation}\label{eq0}
d([x,y])=[d(x),y] + [x, d(y)].
\end{equation}
The set of all derivations of $L$ (denoted by $\Der(L)$) forms a Lie algebra with respect to the commutator.

Note that the operator of right multiplication on elements $x\in L$ (further denoted by $\mathcal{R}_x$) is a derivation, which is called {\it inner derivation}.

\begin{defn} A Leibniz algebra $L$ is called \textit{complete} if $\Center(L)=0$ and all derivations of $L$ are
inner. \end{defn}

Analogously to Lie algebras, we define the following sequences:
\[L^1=L, \ L^{k+1}=[L^k,L],  \ k \geq 1, \qquad \qquad
L^{[1]}=L, \ L^{[s+1]}=[L^{[s]},L^{[s]}], \ s \geq 1,\]
so-called the  \emph{lower central} and the \emph{derived series} of $L$, respectively.

\begin{defn} A Leibniz algebra $L$ is
\emph{nilpotent} (respectively, \emph{solvable}), if there exists $n\in\mathbb N$ ($m\in\mathbb N$) such that $L^{n}=0$ (respectively, $L^{[m]}=0$).
\end{defn}

The maximal nilpotent ideal of a Leibniz algebra is said to be the \emph{nilradical} of the  algebra.

An analogue of Mubarakzjanov's methods has been applied for solvable Leibniz algebras which shows the importance of the consideration of non-characteristically nilpotent Leibniz algebra \cite{Nulfilrad}.

Consider a solvable Leibniz algebra $R=N\oplus Q$  with the nilradical $N$ and complementary vector space $Q$ of $N$ with a basis $\{x_1, \dots, x_m\}$. It is known that for an element $x \in Q$ the operator $\mathcal{R}_{{x |}_{N}}$ is a non-nilpotent derivation of $N$. Moreover, for any scalars $\{\alpha_1, \dots, \alpha_m\}\in \mathbb{C}\setminus\{0\}$, the operator $\alpha_1\mathcal{R}_{{x_1 |}_{N}}+\dots+\alpha_m\mathcal{R}_{{x_m|}_{N}}$ is non-nilpotent,
which means that the elements $\{x_1, \dots, x_m\}$ are {\it nil-independent}. Therefore, the dimension of complementary vector space to $N$ is no greater than the maximal number of nil-independent derivations of $N$ (\cite[Theorem 3.2]{Nulfilrad}).

For a nilpotent Leibniz algebra $L$ and $x\in L\setminus L^2$ we consider the decreasing sequence $C(x)=(n_1,n_2,
\dots,n_k)$ as the dimensions of the Jordan blocks of the operator $\mathcal{R}_x$. On the set of such sequences we consider lexicographic order.

\begin{defn}  The sequence $C(L)=\max\limits_{x\in L \setminus L^2}C(x)$
is called the \emph{characteristic sequence} of the Leibniz algebra $L$.
\end{defn}

Similar to the Lie algebras, we have the following definition.

\begin{defn} A Leibniz algebra $L$ is called $p$-\emph{filiform}  if
$C(L)=(n-p,\underbrace{1,\dots,1}_{p})$, where $p\geq 0$.
\end{defn}

Note that above definition, when $p>0$ agrees with the definition of $p$-filiform Lie algebras \cite{Gomez}.
Since in the case of Lie algebras there is no singly-generated algebra, the notion of $0$-filiform algebra for Lie algebras has no sense,
 while for the Leibniz algebras case in each dimension there exists up to isomorphism a unique null-filiform algebra \cite{Omir1}.

\begin{defn}  Given an $n$-dimensional $p$-filiform Leibniz algebra $L$, put $L_i=L^i/L^{i+1},\ 1\leq i\leq n-p$, and $\gr L =L_1 \oplus L_2 \oplus \cdots \oplus L_{n-p}$.
 Then $[L_i,L_j]\subseteq L_{i+j}$ and we obtain the graded algebra $\gr L$. If $\gr L$ and $L$ are isomorphic, $\gr L\cong L$, we say that $L$ is \emph{naturally graded}.
\end{defn}

In this paper, we consider  naturally graded $p$-filiform non-Lie Leibniz algebras. Their classification is given in the next theorem.

\begin{thm}\label{p-filiform} \cite{p-filLeibniz}  An arbitrary $n$-dimensional naturally graded non-split non-Lie $p$-filiform Leibniz algebra $(n-p\geq 4)$ is isomorphic to one of the following non-isomorphic algebras:

$p=2k$ is even
$$\begin{array}{ll}
\mu_1:\left\{\begin{array}{ll}
[e_i,e_1]=e_{i+1}, & 1\leq i\leq n-2k-1,\\[1mm]
[e_1, f_j] =f_{k+j}, & 1\leq j\leq k,\\[1mm]
 \end{array}\right.&

\mu_2:\left\{\begin{array}{ll}
[e_i,e_1]=e_{i+1}, & 1\leq i\leq n-2k-1,\\[1mm]
[e_1,f_1]=e_{2}+f_{k+1}, &\\[1mm]
[e_i,f_1]=e_{i+1}, & 2\leq i\leq n-2k-1,\\[1mm]
[e_1, f_j] =f_{k+j}, & 2\leq j\leq k,\\[1mm]
 \end{array}\right.
 \end{array}$$

$p=2k+1$ is odd

\[\mu_3:\left\{\begin{array}{ll}
[e_i,e_1]=e_{i+1}, & 1\leq i\leq n-2k-2,\\[1mm]
[e_1,f_j]=f_{k+1+j}, & 1\leq j\leq k,\\[1mm]
[e_i, f_{k+1}] =e_{i+1}, & 1\leq i\leq n-2k-2,\\[1mm]
 \end{array}\right.\]
where $\{e_1,e_2,\dots,e_{n-p},f_1,f_2,\dots,f_{p}\}$ is a basis
of the algebra.
\end{thm}

In order to simplify our next  calculations, the following change of basis in $\mu_3$:
\[e_1^\prime=f_{k+1}, \quad e_2^\prime=e_1-f_{k+1}, \quad e_{i+1}^\prime=e_{i}, \  2\leq i\leq n-2k-1,\quad f_{j}^\prime=f_{j}, \quad f_{k+j}^\prime=f_{k+1+j}, \ 1\leq j\leq k,
\]
allows to obtain a more convenient form of $\mu_3$:
\[ \mu_3:\left\{\begin{array}{ll}
[e_i,e_1]=e_{i+1}, & 2\leq i\leq n-2k-1,\\[1mm]
[e_2,f_j]=f_{k+j}, & 1\leq j\leq k.\\[1mm]
 \end{array}\right.\]

\subsection{Cohomology Leibniz algebras}
Since in the last section of this paper we study the cohomological rigidity of obtained algebras, we need some concepts of the second cohomology group of Leibniz algebras. For more details, we refer to  \cite{Loday}, \cite{Loday1} and references therein. The \emph{second cohomology group} of a Leibniz algebra $L$ with coefficient itself is the quotient space
$$HL^2(L,L):=ZL^2(L,L)/BL^2(L,L),$$
where the elements $\psi\in BL^2(L,L)$ and $\varphi \in ZL^2(L,L)$ are defined by:
$$\psi(x,y)=[d(x),y] + [x,d(y)] - d([x,y]), \,\, \mbox{for some linear map} \,\, d\in Hom(L,L),$$
\begin{equation}\label{eq4}
[x,\varphi(y,z)] - [\varphi(x,y), z] +[\varphi(x,z), y] +\varphi(x,[y,z]) - \varphi([x,y],z)+\varphi([x,z],y)=0,
\end{equation}
respectively.

It is obvious that a Leibniz $2$-cocycle $\varphi$ of a Leibniz algebra $L$ is  determined by the identity
$\Phi(\varphi)(x, y, z)=0,$ where
$$\Phi(\varphi)(x, y, z)=[x,\varphi(y,z)] - [\varphi(x,y), z] +[\varphi(x,z), y] +\varphi(x,[y,z]) - \varphi([x,y],z)+\varphi([x,z],y).$$

\begin{defn} A Leibniz algebra $L$ is called cohomologically rigid if $HL^2(L,L)=0.$
\end{defn}

Due to results of the paper \cite{Omir1}, we have that a Leibniz algebra is rigid if the second cohomology group with coefficients in itself is trivial.

\section{Solvable Leibniz algebras with abelian nilradical and maximal dimension of complemented space $Q.$}

In this section we recall some results of the paper \cite{Adashev}, which will be used below.

We denote by $\mathbf{a_k}$ the $k$-dimensional abelian algebras
and by $R(\mathbf{a_k},s)$ the solvable Leibniz algebra with $\mathbf{a_k}$ as nilradical and $s$ as the dimension of complemented space to $\mathbf{a_k}.$

\begin{thm} \cite{Adashev} \label{abelian} The maximal possible dimension of algebras of the family $R(\mathbf{a_k}, s)$ is equal to $2k$, that is, $s=k$.
Moreover, an arbitrary algebra of the family $R(\mathbf{a_k},
k)$ is decomposed into a direct sum of copies of two-dimensional
non-trivial solvable Leibniz algebras.
\end{thm}

Consider the solvable Leibniz algebras ${\mathcal L(\gamma_i)}$  with nilradical $\mathbf{a_k}$ under the
condition that the complemented space to  the nilradical have maximal dimension. Then there exists a basis $\{f_1, f_2, \dots,
f_k, x_1, x_2, \dots, x_{k}\}$  of ${\mathcal L(\gamma_i)}$ such that the
multiplication table has the form:
\[{\mathcal L(\gamma_i)}:\quad [f_i,x_i] = f_i, \qquad [x_i,f_i]=\gamma_if_i, \quad  1\leq i \leq k,\]
where $\gamma_i\in\{-1,0\}.$

The algebra ${\mathcal L(\gamma_i)}$ is
a rigid algebra for any  $\gamma_i\in\{-1,0\}, \ 1\leq i \leq k$, \cite{Adashev}.

\begin{lem} \label{lemm1} Any automorphism $\varphi$ of the algebra ${\mathcal L(\gamma_i)}$ has the following form:
\[\varphi(f_i)=\alpha_if_i,\ \ \varphi(x_i)=\beta_if_i+x_i, \quad 1\leq i\leq k,\]
where $(1+\gamma_i)\beta_i=0$ for $1\leq i\leq k.$

\end{lem}

\begin{proof} Let $\varphi$ be an automorphism of ${\mathcal L(\gamma_i)}.$ Since the automorphism of algebra maps nilradical to  nilradical  we can assume
$$\varphi(f_i)=\sum\limits_{j=1}^{k}D_{i,j}f_j, \quad \varphi(x_i)=\sum\limits_{j=1}^{k}F_{i,j}f_j+\sum\limits_{j=1}^{k}H_{i,j}x_j,\quad 1\leq i\leq k.$$

From the following equalities  $[\varphi(f_i),\varphi(x_j)]=\varphi([f_i,x_j])$ and $[\varphi(x_i),\varphi(x_j)]=\varphi([x_i,x_j])$ with $1\leq i\leq k,$  we derive
$$\left\{\begin{array}{ll}
D_{i,m}H_{j,m}=0,  & 1\leq i\neq j,m\leq k,\\[1mm]
D_{i,m}H_{i,m}=D_{i,m},  & 1\leq i,m\leq k,\\[1mm]
F_{i,m}H_{j,m}+\gamma_mF_{j,m}H_{i,m}=0,  & 1\leq i,j,m\leq k.\\[1mm]
\end{array}\right.$$

So, for a given value of $j$ we have a linear system with respect to $H_{j,1}, \ H_{j,2}, \ \dots,\ H_{j,k}$.

Let us prove that for a fixed $j, \ 1\leq j\leq k$ there exists only $m_0$ such that
$H_{j,m_0}=1$ and $H_{j,m}=0$ with $1\leq m\neq m_0\leq k.$

Let us suppose that $H_{j,m_0}=H_{j,m_1}=1$, then we get $D_{i,m_0}=D_{i,m_1}=0$ for $1\leq i\neq j\leq k.$ On the other hand, $det(D_{i,m})_{i,m=1}^k=0,$ that is, we arrive at contradiction. Without loss of generality, we can assume that $H_{j, j} =1 $ and $ H_{j, i} = 0 $ with $ 1 \leq j\neq i \leq k $.

Then, we obtain the following restrictions:
$$\left\{\begin{array}{lll}
D_{j,j}\neq0,  & D_{j,m}=0,  & 1\leq m\neq j\leq k,\\[1mm]
F_{i,j}=0,  & (1+\gamma_{j})F_{j,j}=0,  & 1\leq i\neq j\leq k,\\[1mm]
\end{array}\right.$$
which imply
$$\varphi(f_i)=D_{i,i}f_i, \quad \varphi(x_i)=F_{i,i}f_i+x_i,\quad 1\leq i\leq k,$$
where $(1+\gamma_i)F_{i,i}=0$ and $\gamma_i\in\{-1,0\}$ for $1\leq i\leq k.$
\end{proof}

\section{solvable leibniz algebras with $n$-dimensional naturally graded $p$-filiform non-Lie Leibniz algebra and maximal dimension of $Q.$}

In this section we give a description of  solvable Leibniz algebras whose nilradical is a naturally graded $p$-filiform Leibniz algebra and the dimension of $Q$ is maximal. Firstly, we recall the derivations of the algebras $\mu_i,\ i=1,2,3$ given in \cite{Adashev}.

\subsection{Derivations of algebras $\mu_i,\ i=1,2,3$}

\begin{prop} \label{prop1} Any derivation of the algebra $\mu_1$ has the following matrix form:
\[\mathbb{D}=\begin{pmatrix}
A&B\\
C&D
\end{pmatrix},\quad with \quad  D=\begin{pmatrix}
D_1&D_2\\
0&a_1\mathbb{E}+D_1
\end{pmatrix},\]
where
\[A=\sum_{i=1}^{n-2k}ia_{1}e_{i,i}+\sum_{i=1}^{n-2k-1}\sum_{j=i+1}^{n-2k}a_{j-i+1}e_{i,j}, \quad B=\sum_{i=1}^{2k}b_{i}e_{1,i}+\sum_{i=1}^{k}b_{i}e_{2,k+i}, \quad C=\sum_{i=1}^{k}c_ie_{i,n-2k},\]
$\ A\in M_{n-2k,n-2k}, \ B\in M_{n-2k,2k}, \ C\in M_{2k,n-2k}, \ D_1, D_2, \mathbb{E} \in M_{k,k}$ and matrix units $e_{i,j}$.
\end{prop}

\begin{prop} \label{prop2} Any derivation of the algebra $\mu_2$ has the following matrix form:
\[\mathbb{D}=\begin{pmatrix}
A&B\\
C&D
\end{pmatrix}, \quad with \quad  D=\begin{pmatrix}
D_1&D_2\\
0&D_3
\end{pmatrix},\]
where
\[A=\sum_{i=1}^{n-2k}(ia_{1}+(i-1)b_1)e_{i,i}+\sum_{i=1}^{n-2k-1}\sum_{j=i+1}^{n-2k}a_{j-i+1}e_{i,j},\quad B=\sum_{i=1}^{2k}b_{i}e_{1,i}+\sum_{i=1}^{k}b_{i}e_{2,k+i},\]
 \[C=\sum_{i=1}^{k}c_ie_{i,n-2k}, \quad D_1=\sum_{i=1}^{k}\sum_{j=2}^{k}d_{i,j}e_{i,j}+(a_1+b_{1})e_{1,1},\quad D_3=D_1+a_1\mathbb{E}-\sum_{j=1}^{k}b_{j}e_{1,j},\]
with
$\ A\in M_{n-2k,n-2k}, \ B\in M_{n-2k,2k},  \ C\in M_{2k,n-2k}, \ D_1, D_2, D_3, \mathbb{E} \in M_{k,k}$
and matrix units $e_{i,j}$.
\end{prop}

\begin{prop} \label{prop3}  Any derivation of the algebra $\mu_3$ has
the following matrix form:
\[\mathbb{D}=\begin{pmatrix}
A&B\\
C&D
\end{pmatrix}, \quad with \quad  D=\begin{pmatrix}
D_1&D_2\\
0&a_2\mathbb{E}+D_1
\end{pmatrix},\]
where
\[A=a_1e_{1,1}+\sum_{i=2}^{n-2k}((i-2)a_{1}+a_2)e_{i,i}+\beta e_{1,n-2k}+\sum_{i=2}^{n-2k-1}\sum_{j=i+1}^{n-2k}a_{j-i+2}e_{i,j},\]
\[B=\sum_{i=1}^{2k}b_{1,i}e_{1,i}+\sum_{i=1}^{k}b_{2,i}e_{2,k+i}+\sum_{i=1}^{k}b_{1,i}e_{3,k+i}, \quad C=\sum_{i=1}^{k}c_ie_{i,n-2k},\]
with $A\in M_{n-2k,n-2k}, \ B\in M_{n-2k,2k}, \ C\in M_{2k,n-2k}, \ D_1,D_2, \mathbb{E} \in M_{k,k}$ and
matrix units $e_{i,j}$.
\end{prop}

The theorem bellow describes the maximal dimensions of the complemented space to $\mu_i,\ i=1,2,3$.

\begin{thm} \label{thm5} Let $R$ be a solvable Leibniz algebra whose nilradical is $\mu_i,\ i=1,2,3$.
Then the dimension of complemented space to nilradical verifies that:
$$dim\ Q(\mu_i)\leq k+2\left[\frac{i}{3}\right].$$

\end{thm}
\begin{proof} According to Propositions \ref{prop1} and \ref{prop2}, we have the following expresions for $R(\mu_1, s)$ and $R(\mu_2, s)$, respectively:

$$\begin{array}{l}
\left\{\begin{array}{ll}
[e_1,x]=\sum\limits_{i=1}^{n-2k}a_ie_i+\sum\limits_{i=1}^{2k}b_if_i,&\\[1mm]
[e_2,x]=2a_1e_2+\sum\limits_{i=3}^{n-2k}a_{i-1}e_i+\sum\limits_{i=1}^{k}b_if_{k+i},&\\[1mm]
\end{array}\right.\quad \left\{\begin{array}{ll}
[e_1,x]=\sum\limits_{i=1}^{n-2k}a_{i}e_i+\sum\limits_{i=1}^{2k}b_{i}f_i,&\\[1mm]
[e_2,x]=(2a_{1}+b_{1})e_2+\sum\limits_{i=3}^{n-2k}a_{i-1}e_i+\sum\limits_{i=1}^{k}b_{i}f_{k+i},&\\[1mm]
\end{array}\right.
\end{array}$$

Let us introduce the following notations:
\[[x,e_1]=\sum\limits_{i=1}^{n-2k}\beta_ie_i+\sum\limits_{i=1}^{2k}\beta_{n-2k+i}f_i, \qquad \quad
[x,f_1]=\sum\limits_{i=1}^{n-2k}\gamma_{i}e_i+\sum\limits_{i=1}^{2k}\varphi_{i}f_i.\]

The equalities ${\mathcal L}(x, f_1, e_1)={\mathcal L}(e_1, x, e_1)=0$ imply $a_1=0.$

Note that $\{f_1,f_2,\dots,f_k\}$ form the algebra $\mathbf{a_k}$ and the space of derivations of the algebra
$\mathbf{a_k}$ coincided with $M_{k,k}.$

It is easy to see that
$$\mathcal{R}_{{x|}_{\mathbf{a_k}}}\circ\mathcal{R}_{{y
|}_{\mathbf{a_k}}}=\mathcal{R}_{{y |}_{\mathbf{a_k}}}\circ\mathcal{R}_{{x|}_{\mathbf{a_k}}}$$ for any $x,y\in
Q$. This implies that all operators $\mathcal{R}_{{x_i
|}_{\mathbf{a_k}}}, \ 1\leq i \leq s$ could be simultaneously
transformed to their Jordan forms by a basis transformation.
Therefore, the matrix operator $\mathcal{R}_{{x |}_{\mathbf{a_k}}}$ (in our case $\mathcal{R}_{{x |}_{\mathbf{a_k}}}=D_1$)
has the following form:
$$D_1=\left(\begin{array}{cccccc}
 d_{1,1}&d_{1,2}&0&\dots&0&0\\
0&d_{2,2}&d_{2,3}&\dots&0&0\\
 0& 0& d_{3,3}&\dots&0&0\\
\vdots&\vdots& \vdots&\vdots &\vdots&\vdots \\
0& 0& 0&\dots& d_{k-1,k-1}& d_{k-1,k}\\
0& 0& 0&\dots& 0& d_{k,k}\\
\end{array}\right),$$
where $d_{i,i+1}\in\{0,1\}$ for $1\leq i\leq k-1.$

Now we are going to investigate the nilpotency of matrix $\mathbb{D}.$ Due to Propositions \ref{prop1}-\ref{prop3} the nilpotency of $\mathbb {D}$ depends on the matrices $A$ and $D_1.$

Let us consider the matrix $\mathbb {D}$ as follows
$$\mathbb{D}=\left(\begin{array}{cc}
A&B\\
C&D\\
\end{array}\right)=\left(\begin{array}{cc}
A_1+A_2&B\\
C&K_1+K_2\\
\end{array}\right),$$
where $A_1, K_1$ are diagonal matrices and $A_2, K_2$ are nilpotent such that
$$A_1=\left\{\begin{array}{ll}
diag\{0,0,0,\dots,0\},&\mbox{for} \quad \mu_1,\\[1mm]
diag\{0,b_1,2b_1,\dots,(n-2k-1)b_1\},&\mbox{for} \quad \mu_2,\\[1mm]
diag\{a_1,a_2,a_1+a_2,\dots,(n-2k-2)a_1+a_2\},&\mbox{for} \quad \mu_3,\\[1mm]
\end{array}\right.$$
$$K_1=\left\{\begin{array}{ll}
diag\{d_{1,1},d_{2,2},\dots,d_{k,k},d_{1,1},d_{2,2},\dots,d_{k,k}\},&\mbox{for} \quad \mu_1,\\[1mm]
diag\{b_{1},d_{2,2},\dots,d_{k,k},0,d_{2,2},\dots,d_{k,k}\},&\mbox{for} \quad \mu_2,\\[1mm]
diag\{d_{1,1},d_{2,2},\dots,d_{k,k},a_2+d_{1,1},a_2+d_{2,2},\dots,a_2+d_{k,k}\},&\mbox{for} \quad \mu_3.\\[1mm]
\end{array}\right.$$

It is easy to see that $CB=0$ and the matrices $A_1A_2, \ A_2^2,\ BC,\ K_1K_2, \ K_2^2$ are nilpotent.

Moreover, matrices $C(A_1+A_2), \ (K_1+K_2)C$ have the type of $C$ and matrices $(A_1+A_2)B, \ B(K_1+K_2)$ have the type of $B$.

According to the above arguments we have the following recurrence formula:
$$\mathbb{D}^t=\left(\begin{array}{cc}
A_1^t+\widetilde{A}_{2}&\widetilde{B}\\
\widetilde{C}&K_1^{t}+\widetilde{K}_2\\
\end{array}\right), \qquad t\geq1, $$
where $\widetilde{A}_{2}, \ \widetilde{K}_2-$ nilpotent matrices  and matrices $\widetilde{B},\ \widetilde{C}$ have the types of $B$ and $C,$ respectively.

To sum up, we conclude that the matrix $\mathbb{D}$ is nilpotent if and only if $A_1$ and $K_1$ are nilpotents.

Therefore, we obtain the following conclusions:
\begin{itemize}
\item For $\mu_1$, the nilpotency of $\mathbb{D}$ depends on
$d_{i, i},\ 1\leq i\leq k$, that is $\mathbb{D}$ nilpotent if only if $d_{i,i}=0, \ 1\leq i\leq k$.

\item For $\mu_2$, the nilpotency of $\mathbb{D}$ depends on
$b_1$ and $d_{i,i}, \ 2\leq i\leq k$, that is $\mathbb{D}$ nilpotent if only if $b_1=d_{i,i}=0, \ 2\leq
i\leq k$.

\item For $\mu_3,$ the nilpotency of $\mathbb{D}$ depends on
$a_1, a_2$ and $d_{i,i}, \ 1\leq i\leq k$, that is $\mathbb{D}$ nilpotent if only if $a_1=a_2=d_{i,i}=0, \ 1\leq
i\leq k$.
\end{itemize}

Applying the  result in \cite[Theorem 3.2]{Nulfilrad}, the stated inequalities follow.
\end{proof}

The following results will be used in the description of solvable Leibniz algebras whose nilradicals are $\mu_i,\ i=1,2,3$ and with maximal dimensional complemented space of nilradicals.

\begin{prop} \label{annulator}
Let $R$ be a solvable Leibniz algebra whose nilradical is a naturally graded $p$-filiform non-Lie Leibniz algebra. Then
$$\{e_1,f_1,\dots,f_k\}\cap\Ann_r(R)=0 \quad \mbox{and} \quad \{e_2,\dots, e_{n-2k}, f_{k+1},\dots,f_{2k}\}\subseteq \Ann_r(R),$$
with $a_2\neq0$ for the algebra $R(\mu_3,s).$
\end{prop}

\begin{proof}

Using Theorem \ref{p-filiform} and the properties of the right annihilator (that is, $[x,x],[x,y]+[y,x]\in \Ann_r(R)$) the assertion easily follows for $R(\mu_1,s)$ and $R(\mu_2,s).$

Consider the algebra $R(\mu_3,s).$ It is easy to see that
$$e_1,f_1,\dots,f_k\notin \Ann_r(R) \quad \mbox{and} \quad e_3,\dots, e_{n-2k}, f_{k+1},\dots,f_{2k}\in \Ann_r(R).$$

Let us suppose $a_2\neq 0$. Then, from the derivation of $\mu_3$ we get
\[[e_2,x]=\sum_{i=2}^{n-2k}a_ie_{i}+\sum_{i=1}^{k}b_{2,i}f_{k+i}, \quad  [x,e_2]=\sum_{i=1}^{n-2k}\alpha_ie_{i}+\sum_{i=1}^{2k}\beta_{i}f_{i}.\]

The equality $\mathcal{L}(x,e_2,e_1)=0$ implies  $\alpha_i=0$ for $2\leq i\leq n-2k-1.$ Since $[e_2,x]+[x,e_2]\in \Ann_r(R)$ and  $a_2\neq 0$,
 we have $e_2\in \Ann_r(R)$ which complete the proof.
 \end{proof}

\begin{lem} \label{maxabelian} Let $R$ be a solvable Leibniz algebra whose nilradical is a naturally graded $p$-filiform Leibniz algebra.
Then the maximal solvable Leibniz subalgebra with nilradical $\mathbf{a_k}=<f_1,\dots,f_k>$ of $R$ is isomorphic to  $\mathcal{L}(\gamma_i)$ with $\gamma_i=-1$ for $1\leq i\leq k$.
\end{lem}

\begin{proof} Clearly, $\mathbf{a_k}=\{f_1,f_2,\dots,f_k\}$ forms an abelian subalgebra of $R$. By Theorem \ref{abelian} the maximal solvable Leibniz algebra with nilradical $\mathbf{a_k}$ is isomorphic to $\mathcal{L}(\gamma_i).$
Since $[f_i,x_i]+[x_i,f_i]\in \Ann_r(R)$ with $ 1\leq i\leq k$ and $f_i\notin \Ann_r(R)$ with $ 1\leq i \leq k$, the proof of lemma is complete.
\end{proof}
%

In the following theorem we present the description of algebras of the family $R(\mu_1, k)$.

\begin{thm} \label{thmmu1} An arbitrary algebra of the family $R(\mu_1, k)$ admits a basis such that the non-vanishing Leibniz brackets become:
\[R(\mu_1, k)(a_{i,j},\varphi_{i,j},\delta_{i,j}): \quad \left\{\begin{array}{ll}
[e_i,x_j]=\sum\limits_{t=i+1}^{n-2k}a_{t-i+1,j}e_{t},&1\leq i\leq n-2k, \ 1\leq j\leq k,\\[1mm]
[f_i,x_i]=f_i,&1\leq i\leq k,\\[1mm]
[f_{k+i},x_i]=f_{k+i},&1\leq i\leq k,\\[1mm]
[x_i,f_i]=-f_i,&1\leq i\leq k,\\[1mm]
[x_i,f_j]=\varphi_{i,j}f_{k+j},&1\leq i\neq j\leq k,\\[1mm]
[x_i,x_j]=\delta_{i,j}e_{n-2k},&1\leq i, j\leq k.\\[1mm]
\end{array}\right.\]
\end{thm}

\begin{proof} According to Propositions \ref{prop1}, \ref{annulator},  Theorem \ref{thm5} and Lemma \ref{maxabelian} we have the following brackets for
$R(\mu_1, k)$:
$$\left\{\begin{array}{ll}
[e_1,x_i]=\sum\limits_{t=2}^{n-2k}a_{t,i}e_t+\sum\limits_{t=1}^{2k}b_{t,i}f_t,&1\leq i\leq k,\\[1mm]
[e_2,x_i]=\sum\limits_{i=3}^{n-2k}a_{t-1,i}e_t+\sum\limits_{t=1}^{k}b_{t,i}f_{k+t},&1\leq i\leq k,\\[1mm]
[e_j,x_i]=\sum\limits_{t=j+1}^{n-2k}a_{t-j+1,i}e_t,&3\leq j\leq n-2k,\\[1mm]
[f_i,x_i]=c_{i,i}e_{n-2k}+f_i+\sum\limits_{t=k+1}^{2k}d_{i,i}^tf_t,&1\leq i\leq k,\\[1mm]
[f_i,x_j]=c_{i,j}e_{n-2k}+\sum\limits_{t=k+1}^{2k}d_{i,j}^tf_t,&1\leq i\neq j\leq k,\\[1mm]
[f_{k+i},x_i]=f_{k+i},&1\leq i\leq k,\\[1mm]
[x_i,e_1]=\sum\limits_{t=2}^{n-2k}\beta_{t,i}e_t-\sum\limits_{t=1}^{k}b_{t,i}f_t+\sum\limits_{t=1}^{k}\beta_{n-k+t,i}f_{k+t},&1\leq i\leq k,\\[1mm]
[x_i,f_i]=\sum\limits_{t=2}^{n-2k}\gamma_{i,i}^te_t-f_i+\sum\limits_{t=k+1}^{2k}\varphi_{i,i}^tf_t,&1\leq i\leq k,\\[1mm]
[x_i,f_j]=\sum\limits_{t=2}^{n-2k}\gamma_{i,j}^te_t+\sum\limits_{t=k+1}^{2k}\varphi_{i,j}^tf_t,&1\leq i\neq j\leq k,\\[1mm]
[x_i,x_j]=\sum\limits_{t=1}^{n-2k}\delta_{i,j}^te_t+\sum\limits_{t=k+1}^{2k}\theta_{i,j}^tf_t,& 1\leq i,j\leq k.\\[1mm]
\end{array}\right.$$

By taking the change of basis
$$\begin{array}{lll}
e_1^\prime=e_1-\sum\limits_{t=1}^{k}b_{t,t}f_{t}, & &
e_2^\prime=e_2-\sum\limits_{t=1}^{k}b_{t,t}f_{k+t},\\[5mm]

f_i^\prime=f_i-\gamma_{i,i}e_{n-2k}-\sum\limits_{t=k+1}^{2k}\varphi_{i,i}^tf_t, & & x_i^\prime=x_i-\sum\limits_{t=1}^{k}\theta_{i,t}^{k+t}f_{k+t}, \quad 1\leq i\leq k,
\end{array}$$
we can assume
$b_{i,i}=\gamma_{i,i}^{n-2k}=\varphi_{i,i}^{k+t}=\theta_{i,t}^{k+t}=0,$ $1\leq i,t\leq k.$

Applying the Leibniz identity, the following relations are obtained
$$\left\{\begin{array}{lll}
{\mathcal L}(x_i,f_j,e_1)=0, & \Rightarrow & \gamma_{i,j}^t=0, \qquad  1\leq i, j\leq k, \qquad
2\leq t\leq n-2k-1,\\[1mm]
{\mathcal L}(e_1,x_i,x_j)=0,& \Rightarrow & b_{i,j}=b_{k+i,j}=\delta_{i,i}^1=\delta_{i,j}^1=0, \qquad 1\leq i\neq j\leq k. \\[1mm]
\end{array}\right.$$

Putting
$e_1^\prime=e_1-\sum\limits_{t=1}^{k}b_{k+t,t}f_{k+t}, \quad x_i^\prime=x_i-\sum\limits_{t=2}^{n-2k}\beta_{t,i}e_{t-1},$
we conclude that $\beta_{t,i}=b_{k+i,i}=0$ with $ 1\leq i\leq k, \ 2\leq t\leq n-2k.$

Considering the Leibniz identity, we obtain  the following restrictions on structure constants:
$$\left\{\begin{array}{llll}
{\mathcal L}(x_i,f_i,x_i)=0, & \Rightarrow & c_{i,i}=d_{i,i}^t=0,&  1\leq i\leq k, \quad k+1\leq t\leq 2k,\\[1mm]
{\mathcal L}(x_i,f_i,x_j)=0,& \Rightarrow &c_{i,j}=d_{i,j}^t=0,& 1\leq i\neq j\leq k, \quad  k+1\leq t\leq 2k, \\[1mm]
{\mathcal L}(x_i,f_j,x_j)=0,& \Rightarrow & \gamma_{i,j}^{n-2k}=\varphi_{i,j}^{k+t}=0,& 1\leq i\neq j\neq t\leq k, \\[1mm]
{\mathcal L}(x_i,e_1,x_j)=0,& \Rightarrow & \delta_{i,j}^t=\beta_{n-k+i,j}=0, & 1\leq i,j\leq k,\quad 2\leq t\leq n-2k-1, \\[1mm]
{\mathcal L}(x_i,x_j,x_s)=0,& \Rightarrow & \theta_{i,j}^{k+s}=0,&  1\leq i,j\neq s\leq k. \\[1mm]
\end{array}\right.$$

\end{proof}

Below the necessary and sufficient conditions of the existence of an isomorphism between two algebras of the family
$R(\mu_1, k)(a_{i,j},\varphi_{i,j},\delta_{i,j})$ are established.

\begin{prop} \label{propmu1} Two algebras $R(\mu_1, k)^\prime(a_{i,j}^\prime,\varphi_{i,j}^\prime,\delta_{i,j}^\prime)$ and $R(\mu_1, k)(a_{i,j},\varphi_{i,j},\delta_{i,j})$ are isomorphic if and only if there exists $A\in \mathbb{C}^*$ such that
$$\begin{array}{llll}
a_{i,j}^\prime=\frac{a_{i,j}}{A^{i-1}}, & 2\leq i\leq n-2k+1, & 1\leq j\leq k,&\\
\varphi_{i,j}^\prime=\frac{\varphi_{i,j}}{A}, &1\leq i\neq
j\leq k,& \delta_{i,j}^\prime=\frac{\delta_{i,j}}{A^{n-2k}},
& 1\leq i,j\leq k.
\end{array}$$

\end{prop}

\begin{proof}  Taking into account Lemma \ref{lemm1} we consider the general change of generator basis elements of an algebra from $R(\mu_1, k)(a_{i,j},\varphi_{i,j},\delta_{i,j})$:
$$\begin{array}{ll}
e_1^\prime=\sum\limits_{i=1}^{n-2k}A_{i}e_i+\sum\limits_{i=1}^{2k}B_{i}f_i,&
f_i^\prime=\sum\limits_{j=1}^{n-2k}C_{i,j}e_j+D_{i,i}f_i+\sum\limits_{j=k+1}^{2k}D_{i,j}f_j,
\\
x_i^\prime=\sum\limits_{j=1}^{n-2k}E_{i,j}e_j+F_{i,i}f_i+\sum\limits_{j=k+1}^{2k}F_{i,j}f_j+x_i,& \qquad 1\leq i\leq k.
\end{array}$$

From the following brackets in $R(\mu_1,k)^\prime(a_{i,j}^\prime,\varphi_{i,j}^\prime,\delta_{i,j}^\prime):$
$$[e_i^\prime,e_1^\prime]=e_{i+1}^\prime, \ 1\leq i\leq n-2k-1, \quad [f_i^\prime,e_1^\prime]=0, \quad [e_1^\prime,f_i^\prime]=f_{k+i}^\prime, \ 1\leq i\leq k,$$ we derive
$$\begin{array}{ll}
e_{2}^\prime=A_1\sum\limits_{i=2}^{n-2k}A_{i-1}e_i+A_1\sum\limits_{i=1}^{k}B_{i}f_{k+i}, & e_{i}^\prime=A_1^{i-1}\sum\limits_{j=i}^{n-2k}A_{j-i+1}e_j,\
\ 3\leq i\leq n-2k,\\[5mm]
f_{k+i}^\prime=A_1D_{i,i}f_{k+i},& C_{i,j}=0,   \ \ 1\leq i\leq k,\ \ 1\leq j\leq n-2k-1.
\end{array}$$

The following vanishing parameters $B_i=E_{i,j}=C_{i,n-2k}=D_{i,k+t}=0$ with $ 1\leq i,t\leq k$ and $1\leq j\leq n-2k-1$
have been obtained from the products:
$$[x_i^\prime,e_1^\prime]=0, \quad [x_i^\prime,f_i^\prime]=f_i^\prime, \  1\leq i\leq k.$$

Therefore, we obtain
$$\begin{array}{lll}
e_1^\prime=\sum\limits_{i=1}^{n-2k}A_{i}e_i+\sum\limits_{i=k+1}^{2k}B_{i}f_i,&
e_{i}^\prime=A_1^{i-1}\sum\limits_{j=i}^{n-2k}A_{j-i+1}e_j,&
2\leq i\leq n-2k,\\[2mm]
f_i^\prime=D_{i,i}f_i,\ \ f_{k+i}^\prime=A_1D_{i,i}f_{k+i}, & x_i^\prime=E_{i,n-2k}e_{n-2k}+F_{i,i}f_i+\sum\limits_{j=k+1}^{2k}F_{i,j}f_j+x_i,&\quad 1\leq i\leq k.
\end{array}$$

Let us consider the products
$$[e_1^\prime,x_j^\prime]=\sum\limits_{i=2}^{n-2k}a_{i,j}^\prime e_i^\prime, \quad [x_i^\prime,f_j^\prime]=\varphi_{i,j}^\prime
f_{k+j}^\prime, \quad [x_i^\prime,x_j^\prime]=\delta_{i,j}^\prime
e_{n-2k}^\prime, \ 1\leq i, j\leq k.$$

From which we get the following restrictions:
$$\left\{\begin{array}{ll}
a_{i,j}^\prime=\frac{a_{i,j}}{A_1^{i-1}}, & 2\leq i\leq n-2k, \ \ 1\leq j\leq k,\\[2mm]
\varphi_{i,j}^\prime=\frac{\varphi_{i,j}}{A_1}, & 1\leq i\neq j\leq k,\\[2mm]
\delta_{i,j}^\prime=\frac{\delta_{i,j}}{A_1^{n-2k}}, &   1\leq i,j\leq k,\\[2mm]
\end{array}\right.$$
where $A_1F_{j,j}+B_{k+j}=F_{i,k+i}=F_{i,k+j}+\varphi_{i,j}F_{j,j}=0, \quad  1\leq i\neq j\leq k.$
\end{proof}

Below we describe solvable Leibniz algebras $R(\mu_2 , k)$.

\begin{thm} \label{mu2} An arbitrary algebra of the family $R(\mu_2 , k)$ admits a basis such that its  multiplication table has the following form:
$$R(\mu_2, k)(b_i,\beta_i,\varphi_{i,j},\theta_{i,j}):\quad\left\{\begin{array}{ll}
[e_1,x_1]=f_1+b_{1}f_{k+1},&\\[1mm]
[e_2,x_1]=e_2+f_{k+1},&\\[1mm]
[e_j,x_1]=(j-1)e_j,&3\leq j\leq n-2k,\\[1mm]
[x_1,e_1]=-f_1+\beta_{1}f_{k+1},&\\[1mm]
[e_1,x_i]=b_{i}f_{k+1},&2\leq i\leq k,\\[1mm]
[f_i,x_i]=f_i,&1\leq i\leq k,\\[1mm]
[f_{k+i},x_i]=f_{k+i},&2\leq i\leq k,\\[1mm]
[x_i,e_1]=\beta_{i}f_{k+1},&2\leq i\leq k,\\[1mm]
[x_i,f_i]=-f_i,&1\leq i\leq k,\\[1mm]
[x_i,f_j]=\varphi_{i,j}f_{k+j},&1\leq i\leq k,\ 2\leq j\leq k,\ i\neq j,\\[1mm]
[x_i,x_j]=\theta_{i,j}f_{k+1},&1\leq i,j\leq k.\\[1mm]
\end{array}\right.$$
\end{thm}

\begin{proof} The description of $R(\mu_2, k)$ follows from Proposition \ref{prop2}, \ref{annulator}, Theorem \ref{thm5}  and Lemma \ref{maxabelian}. In fact, firstly, we consider derivations of $\mu_2$ and
since the parameters $b_1, d_2,d_3,\dots, d_k$ are in the diagonal, we have only $k$ nil-independent derivations which correspond to the values of $(b_1, d_2,d_3,\dots d_k):$  $(1,0,0,\dots,0), \ (0,1,0,\dots,0),\ \dots, \ (0,0,0,\dots,1)$.
Later, assuming these derivations as $\mathcal{R}_{x_{1}},\ \mathcal{R}_{x_{2}},\ \dots,\ \mathcal{R}_{x_{k}}$ (respectively) we complete the proof by applying similar arguments as used in the proof of Theorem \ref{thmmu1}.
\end{proof}

In the next proposition necessary and sufficient conditions of
the existence of an isomorphism between two algebras of the family
$R(\mu_2, k)(b_i,\beta_i,\varphi_{i,j},\theta_{i,j})$ are established.

\begin{prop} \label{propmu2} Two algebras $R(\mu_2, k)^\prime(b_i^\prime,\beta_i^\prime,\varphi_{i,j}^\prime,\theta_{i,j}^\prime)$ and $R(\mu_2, k)(b_i,\beta_i,\varphi_{i,j},\theta_{i,j})$ are isomorphic if and only if there exists $A\in \mathbb{C}^*$ such that
$$\begin{array}{llll}
b_{i}^\prime=\frac{b_{i}}{A}, & 1\leq i\leq
k,&  \beta_{i}^\prime=\frac{\beta_{i}}{A}, &
1\leq i\leq k,\\[2mm]
\varphi_{1,i}^\prime=\frac{\varphi_{1,i}}{A},
& 2\leq  i\leq k,&
\varphi_{i,j}^\prime=\frac{\varphi_{i,j}}{A}, & 2\leq i\neq
j\leq k,\\ [2mm]
\theta_{i,j}^\prime=\frac{\theta_{i,j}}{A^2}, &
1\leq i,j\leq k.& &
\end{array}$$
\end{prop}

\begin{proof} Analogously to the proof of Proposition \ref{propmu1}.
\end{proof}


To complete the description of solvable Leibniz algebras with the nilradicals $\mu_i, \ i=1, 2, 3$ and maximal complemented space to nilradical, we give the following theorem.

\begin{thm}\label{mu3} An arbitrary algebra of the family $R(\mu_3, k+2)$ admits a basis such that its  multiplication table has the following form:
$$R(\mu_3, k+2):\quad\left\{\begin{array}{ll}
[e_1,y_1]=e_1,&\\[1mm]
[e_j,y_1]=(j-2)e_j,&2\leq j\leq n-2k,\\[1mm]
[y_1,e_1]=-e_1,&\\[1mm]
[e_j,y_2]=e_j,&2\leq j\leq n-2k,\\[1mm]
[f_{k+i},y_2]=f_{k+i},&1\leq i\leq k,\\[1mm]
[f_i,x_i]=f_i,&1\leq i\leq k,\\[1mm]
[f_{k+i},x_i]=f_{k+i},&1\leq i\leq k,\\[1mm]
[x_i,f_i]=-f_{i},&1\leq i\leq k.\\[1mm]
\end{array}\right.$$
\end{thm}

\begin{rem} If in the Theorem \ref{mu3} $k=0$, then we obtain solvable Leibniz algebras with  non-split
naturally graded filiform Leibniz algebra nilradical. Such Leibniz algebra was studied in the work \cite[Theorem 1]{Ladra}.

\end{rem}

\section{Rigidity of the algebra $R(\mu_3,k+2)$.}

Finally, we study the rigidity of the obtained algebras by using triviality of the second cohomology group.

In order to simplify further calculations for the algebra $R(\mu_3,k+2)$, by taking
\[e_{i}^\prime=e_{i+1},  \ 1\leq i\leq n-2k-1, \quad e_{n-2k-1+i}^\prime=f_{k+i},  \ 1\leq i\leq k,\]
\[f_1^\prime=e_1, \quad f_{i}^\prime=f_{i-1},  \ 2\leq i\leq k+1,\]
 \[y_1^\prime=y_1, \quad y_{i}^\prime=x_{i-1},  \ 2\leq i\leq k+1, \quad y_{k+2}^\prime=y_2,\]
we obtain the table of multiplication of the algebras $R(\mu_3,k+2)$ in the following form:
$$\label{ri1}\left\{\begin{array}{ll}
[e_i,f_1]=e_{i+1},& 1\leq i\leq n-2k-2,\\[1mm]
[e_1,f_i]=e_{n-2k+i-2},& 2\leq i\leq k+1,\\[1mm]
[f_i,y_i]=-[y_i,f_i]=f_i,& 1\leq i\leq k+1,\\[1mm]
[e_i,y_{k+2}]=e_i,& 1\leq i\leq n-k-1,\\[1mm]
[e_i,y_{1}]=(i-1)e_i,& 1\leq i\leq n-2k-1,\\[1mm]
[e_{n-2k+i-2},y_{i}]=e_{n-2k+i-2},& 2\leq i\leq k+1.\\[1mm]
\end{array}\right.$$

Moreover, substituting instead $n-k-1$ to $n$ and $k+1$ instead $k$ we obtain
$$\label{ri1} R(\mu_3,k+2)\cong R_n:\left\{\begin{array}{ll}
[e_i,f_1]=e_{i+1},& 1\leq i\leq n-k,\\[1mm]
[e_1,f_i]=e_{n-k+i},& 2\leq i\leq k,\\[1mm]
[f_i,y_i]=-[y_i,f_i]=f_i,& 1\leq i\leq k,\\[1mm]
[e_i,y_{k+1}]=e_i,& 1\leq i\leq n,\\[1mm]
[e_i,y_{1}]=(i-1)e_i,& 1\leq i\leq n-k+1,\\[1mm]
[e_{n-k+i},y_{i}]=e_{n-k+i},& 2\leq i\leq k.\\[1mm]
\end{array}\right.$$

Since $<e_{m+1}, \dots, e_{n-k}>$ with $1\leq m\leq n-k+1$ and $<e_{n-k+m+1}, \dots, e_{n}>$ with $2\leq m\leq k$ form ideals of the algebra $R_n$, we consider the quotient algebras $R_m=R_n/<e_{m+1}, \dots, e_{n-k}>$ and
$R_{n-k+m}=R_n/<e_{n-k+m+1}, \dots, e_{n}>$, which have the following table of multiplications:
$$\label{ri1}R_m:\left\{\begin{array}{ll}
[e_i,f_1]=e_{i+1},& 1\leq i\leq m-1,\\[1mm]
[f_i,y_i]=-[y_i,f_i]=f_i,& 1\leq i\leq k,\\[1mm]
[e_i,y_{k+1}]=e_i,& 1\leq i\leq m,\\[1mm]
[e_i,y_{1}]=(i-1)e_i,& 1\leq i\leq m\\[1mm]
\end{array}\right.$$

$$\label{ri1}R_{n-k+m}:\left\{\begin{array}{ll}
[e_i,f_1]=e_{i+1},& 1\leq i\leq n-k,\\[1mm]
[e_1,f_i]=e_{n-k+i},& 2\leq i\leq m,\\[1mm]
[f_i,y_i]=-[y_i,f_i]=f_i,& 1\leq i\leq k,\\[1mm]
[e_i,y_{k+1}]=e_i,& 1\leq i\leq n-k+m,\\[1mm]
[e_i,y_{1}]=(i-1)e_i,& 1\leq i\leq n-k+1,\\[1mm]
[e_{n-k+i},y_{i}]=e_{n-k+i},& 2\leq i\leq m.\\[1mm]
\end{array}\right.$$

\begin{prop}  Any derivation $d$ of the algebra from $R_n$ has the following form:
\[\left\{\begin{array}{ll}
d(e_1)=ae_1-c_1e_2-\sum\limits_{i=2}^{k}c_ie_{n-k+i},&\\[1mm]
d(e_i)=(a+(i-1)b_1)e_i-c_1e_{i+1},&2\leq i\leq n-k,\\[1mm]
d(e_{n-k+1})=(a+(n-k)b_1)e_{n-k+1},&\\[1mm]
d(e_{n-k+i})=(a+b_i)e_{n-k+i},&2\leq i\leq k,\\[1mm]
d(f_i)=b_if_i,&1\leq i\leq k,\\[1mm]
d(y_i)=c_if_i,&1\leq i\leq k.\\[1mm]
\end{array}\right.\]
\end{prop}
\begin{proof} The proof is carrying out by straightforward verification of derivation property \eqref{eq0}.
\end{proof}

\begin{cor} The solvable Leibniz algebra $R(\mu_3,k+2)$ is complete.
\end{cor}
\begin{proof} The straightforward verifications show that the above derivations of the algebra $R_n$ are linear combinations of $\mathcal{R}_{f_{1}}, \ \mathcal{R}_{f_{2}}, \ \dots, \ \mathcal{R}_{f_{k}}$ and $\mathcal{R}_{y_{1}}, \ \mathcal{R}_{y_{2}}, \ \dots, \ \mathcal{R}_{y_{k+1}}.$ Taking into account the structure of the algebra $R_n$ we get $\Center(R_n)=0.$ The fact that $R(\mu_3,k+2)\cong R_n$ complete the proof of corollary.
\end{proof}

Now we are going to calculate the dimension of $HL^2(R_n,R_n)$.

Let $R$ be an arbitrary Leibniz algebra such that exists an ideal $J$ of the algebra $R$ with the property: $HL^2(L,L)=0$ for the quotient algebra $L=R/J$.

Assuming that $R=L\oplus J$ is a direct sum of the vector spaces $L$ and $J.$ We consider $\varphi\in ZL^2(R,R)$,
then $\varphi$ can be decomposed in the form:
$$\varphi=\varphi_{L,L}^{L}+\varphi_{L,L}^{J}+\varphi_{L,J}^{L}+\varphi_{L,J}^{J}+\varphi_{J,L}^{L}+
\varphi_{J,L}^{J}+\varphi_{J,J}^{L}+\varphi_{J,J}^{J},$$
where
$$\varphi_{A,B}^{C}: A\otimes B \rightarrow C,\quad A,B,C \in \{L,J\}.$$

From equalities $\Phi(\varphi)(x, y, z)=0$ with $x,y,z \in L$ we get
\begin{equation}\label{eq1}proj_{L}\Phi(\varphi_{L,L}^{L})(x,y,z)=0, \quad proj_{J}\Phi(\varphi_{L,L}^{L})(x,y,z)+\Phi(\varphi_{L,L}^{J})(x,y,z)=0.\end{equation}
From (\ref{eq1}) we conclude $\varphi_{L,L}^{L} \in ZL^2(L,L)$. Thus,
\begin{equation}\label{eq2}
ZL^2(R,R)=ZL^2(L,L)+\overline{ZL^2(L,L)},
\end{equation}
where $\overline{ZL^2(L,L)}$ is the complementary space to $ZL^2(L,L)$.

Applying the same arguments one can obtain
\begin{equation*}\label{eq3}
BL^2(R,R)=BL^2(L,L)+\overline{BL^2(L,L)},
\end{equation*}
where $\overline{BL^2(L,L)}$ is the complementary space to $BL^2(L,L)$.

Therefore, the equality $dim\overline{ZL^2(L,L)}=dim\overline{BL^2(L,L)}$ implies $HL^2(R,R)=0$.

Note that the algebra $R_m$ with $1\leq m\leq n-k+1$ is a solvable Leibniz algebra with $p$-filiform nilradical, which is a direct sum of $(m+1)$-dimensional filiform Leibniz algebra with the products $[e_i,f_1]=e_{i+1}, \ 1\leq i\leq m-1$ and $\mathbb{C}^{p-1}$. Moreover, the solvable Leibniz algebra $R_{n-k+m}$ has the same structure as $R_n$ but the dimension is less than $k-m$, that is, $R_{n-k+m}$ is a solvable Leibniz algebra with $(n+m)$-dimensional nilradical (which is isomorphic to $(n+m)$-dimensional algebra $\mu_3$) and $(k+1)$-dimensional complemented space to nilradical.

The triviality of $HL^2(R_n, R_n)$ we shall prove in two steps. First we shall prove that $HL^2(R_m,R_m)=0$ for any $1\leq m\leq n-k+1$, then using this result we shall prove that $HL^2(R_{n-k+m},R_{n-k+m})=0$ for any $2\leq m\leq k$. In both cases we shall use the induction method. The first step follows from the results of paper \cite{Adashev}, where it is proved the triviality of the second cohomology group of the quotient algebra $R_1=\{e_1,f_i,y_i,y_{k+1}\},$ with $1\leq i\leq k$.

Let us assume that $HL^2(R_m,R_m)=0$ for $1\leq m\leq n-k$. Then it is easy to see that $R_{m+1}=R_m\oplus J_{m+1}$ with
the ideal $J_{m+1}=<e_{m+1}>$ and $R_m\simeq R_{m+1}/J_{m+1}.$

Taking into account the equality \eqref{eq2} for $\varphi \in \overline{ZL^2(R_m, R_m)}$ we have

$$\varphi(R_m,R_m)\subseteq J_{m+1}, \ \varphi(R_m,J_{m+1}), \ \varphi(J_{m+1},R_m), \ \varphi(J_{m+1},J_{m+1})\subseteq R_{m+1}.$$

\begin{prop} \label{prop53} The following cochains:

$$\begin{array}{ll}
\varphi_i(e_i,f_1)=e_{m+1},&1\leq i\leq m,\\[1mm]
\psi(e_1,y_1)=e_{m+1}, &\\[1mm]
\phi_i(y_i,f_i)=e_{m+1},&1\leq i\leq k,\\[1mm]
\chi_i(y_i,y_{k+1})=e_{m+1}, &1\leq i\leq k+1,\\[1mm]
\end{array}$$
form a basis of spaces $\overline{ZL^2(R_m, R_m)}$ and $\overline{BL^2(R_m, R_m)}, \ 1\leq m\leq n-k$. 

\end{prop}

\begin{proof} For an arbitrary $\varphi\in \overline{ZL^2(R_m, R_m)}$ by straightforward calculations of equations (\ref{eq4}) on the basis elements of the algebra $R_m$ we derive the following:
\begin{equation}\label{eq1111} \left\{\begin{array}{ll}
\varphi(e_i,y_1)=(i-1-m)\varphi(e_{i-1},f_1),&2\leq i\leq m,\\[1mm]
\varphi(f_1,y_1)=(m-1)\varphi(y_1,f_1),&\\[1mm]
\varphi(f_i,y_1)=m\varphi(y_i,f_i),&2\leq i\leq k,\\[1mm]
\varphi(f_i,y_i)=-\varphi(y_i,f_i),&2\leq i\leq k,\\[1mm]
\varphi(f_i,y_{k+1})=\varphi(y_i,f_i),&1\leq i\leq k,\\[1mm]
\varphi(y_i,y_{1})=m\varphi(y_i,y_{k+1}),&1\leq i\leq k+1\\[1mm]
\end{array}\right.
\end{equation}
with $\{\varphi(e_i,f_1), \varphi(e_1,y_1), \varphi(y_i,f_i), \varphi(y_i,y_{k+1})\}\subseteq J_{m+1}$.

If $\varphi\in \overline{BL^2(R_m, R_m)}$, then substituting to the equality $\varphi(x,y)=[d(x),y]+[x,d(y)]-d([x,y])$ various values of  $x,y\in R_{m+1}$ and for some $d\in C^1(R_{m+1},R_{m+1})$, we get the relations (\ref{eq1111}) on 2-coboundaries $\varphi\in \overline{BL^2(R_m, R_m)}$.
\end{proof}

\begin{cor} \label{cor2} The following holds $$\overline{ZL^2(R_m, R_m)}=\overline{BL^2(R_m, R_m)}.$$
\end{cor}

From Corollary \ref{cor2} we obtain the following result.

\begin{thm} \label{thm57} The algebra $R_m$ is cohomologically rigid for any values of $m, \ 1\leq m\leq n-k+1.$
\end{thm}
\begin{rem} Since the cohomologically rigidness implies rigidness we get that the algebra $R_m$ is rigid for any values of $m, \ 1\leq m\leq n-k+1.$
\end{rem}

Now we are going to prove the triviality of $HL^2(R_n, R_n)$. We shall prove it also by induction.

Let us assume that $HL^2(R_{n-k+m},R_{n-k+m})=0$ for $1\leq m\leq k$. Then it is easy to see that $R_{n-k+m+1}=R_{n-k+m}\oplus J_{m+1}$ with ideal $J_{n-k+m+1}=<e_{n-k+m+1}>$ and $R_{n-k+m}\simeq R_{n-k+m+1}/J_{n-k+m+1}.$

From Theorem \ref{thm57}, we have the first step of the induction method ($m=1$), $HL^2(R_{n-k+1}, R_{n-k+1})=0.$

Applying the same arguments as above we need to check the equality $dim\overline{ZL^2(R_{n-k+m},R_{n-k+m})}=dim\overline{BL^2(R_{n-k+m},R_{n-k+m})}.$

The following results have been tested in a similar way to the results for the algebras $R_m.$

\begin{prop}  The following cochains:
$$\begin{array}{ll}
\varphi_i(e_i,f_1)=e_{n-k+m+1},& 1\leq i\leq n-k, \\[1mm]
\psi(e_1,y_{m+1})=e_{n-k+m+1}, &\\[1mm]
\phi_i(y_i,f_i)=e_{n-k+m+1}, & 1\leq i\leq k,\\[1mm]
\chi_i(y_i,y_{k+1})=e_{n-k+m+1}, & 1\leq i\leq k+1,\\[1mm]
\xi_i(e_1,f_i)=e_{n-k+m+1}, & 2\leq i\leq m+1.\\[1mm]
\end{array}$$

form a basis of spaces $\overline{ZL^2(R_{n-k+m}, R_{n-k+m})}$ and
$\overline{BL^2(R_{n-k+m}, R_{n-k+m})}$ with $1\leq m\leq k-1$. 
\end{prop}
\begin{proof}
Similarly as in the proof of Proposition \ref{prop53} for any $\varphi \in \overline{ZL^2(R_{n-k+m}, R_{n-k+m})}\cup \overline{BL^2(R_{n-k+m}, R_{n-k+m})}$ we derive the following relations:
$$\left\{\begin{array}{ll}
\varphi(e_i,y_1)=(i-1)\varphi(e_{i-1},f_1),&2\leq i\leq n-k+1,\\[1mm]
\varphi(e_i,y_{m+1})=-\varphi(e_{i-1},f_1),&2\leq i\leq n-k+1,\\[1mm]
\varphi(e_{n-k+i},y_i)=\varphi(e_1,f_i),&2\leq i\leq m,\\[1mm]
\varphi(e_{n-k+i},y_{m+1})=-\varphi(e_1,f_i),&2\leq i\leq m,\\[1mm]
\varphi(f_i,y_i)=-\varphi(y_i,f_i),&1\leq i\leq k,\ i\neq m+1,\\[1mm]
\varphi(f_i,y_{m+1})=\varphi(y_i,f_i),&1\leq i\leq k,\ i\neq m+1,\\[1mm]
\varphi(f_i,y_{k+1})=\varphi(y_i,f_i),&1\leq i\leq k,\\[1mm]
\varphi(y_i,y_{m+1})=\varphi(y_i,y_{k+1}),&1\leq i\leq k+1\\[1mm]
\end{array}\right.$$
with $\{\varphi(e_i,f_1), \varphi(e_1,f_i), \varphi(e_1,y_{m+1}), \varphi(y_i,f_i), \varphi(y_i,y_{k+1})\}\subseteq J_{n-k+m+1}.$

\end{proof}

\begin{cor} \label{cor22} $\overline{ZL^2(R_{n-k+m}, R_{n-k+m})}=\overline{BL^2(R_{n-k+m}, R_{n-k+m})}, \ 1\leq m\leq k-1.$
\end{cor}

From Corollary \ref{cor22} and Theorem \ref{thm57} we obtain the main result of this section.

\begin{thm} The algebra $R_{s}$ is a cohomologically rigid algebra for any values of $s \ (1\leq s \leq n).$
\end{thm}

\begin{rem} Since the cohomologically rigidness implies rigidness we get that the algebra $R_s$ is rigid for any values of $s, \ 1\leq s\leq n.$ Therefore, the algebra $R(\mu_3, k+2)$ is a rigid algebra.
\end{rem}

\section*{ Acknowledgements}

The work was partially supported by Ministerio de Econom\'{i}a y Competitividad (Spain), grant MTM2016-
79661-P (European FEDER support included, UE). The first named author was supported by Instituto de Matem\'{a}ticas de la Universidad de Sevilla  by a grant from the Simons Foundation.

\end{document}